\documentclass[reqno,11pt]{amsart}
\usepackage{amsmath}
\usepackage{amsfonts}
\usepackage{amssymb}
\usepackage{graphicx,psfrag,epsfig}
\usepackage{color}
\usepackage{mathrsfs}
\usepackage{pdfsync}
\usepackage{cite}
\usepackage{times}
\usepackage{tikz}
\usepackage{pgfplots}
\pgfplotsset{compat=1.10}
\usepgfplotslibrary{fillbetween}
\usetikzlibrary{patterns}

\newtheorem{thm}{Theorem}[section]
\newtheorem{lem}[thm]{Lemma}
\newtheorem{cor}[thm]{Corollary}
\newtheorem{prop}[thm]{Proposition}

\DeclareMathAlphabet{\mathpzc}{OT1}{pzc}{m}{it}

\numberwithin{equation}{section}

\newcommand{\R}{\mathbb{R}}

\newcommand{\ve}{\varepsilon}

\newcommand{\rd}{\mathrm{d}}

\newcommand{\bqn}{\begin{equation}}
\newcommand{\eqn}{\end{equation}}
\newcommand{\bean}{\begin{eqnarray}}
\newcommand{\eean}{\end{eqnarray}}
\DeclareMathAlphabet{\mathpzc}{OT1}{pzc}{m}{it}
\newtheorem{theorem}{Theorem}[section]

\newtheorem{remark}[theorem]{Remark}
\numberwithin{equation}{section}
\title[Touchdown is the Only Finite Time Singularity in a 3D MEMS Model]{Touchdown is the Only Finite Time Singularity in a Three-Dimensional MEMS Model}

\author{Philippe Lauren\c{c}ot}
\address{Institut de Math\'ematiques de Toulouse, UMR~5219, Universit\'e de Toulouse, CNRS \\ F--31062 Toulouse Cedex 9, France}
\email{laurenco@math.univ-toulouse.fr}
\thanks{Partially supported by the CNRS Projet International de Coop\'eration Scientifique PICS07710}

\author{Christoph Walker}
\address{Leibniz Universit\"at Hannover\\ Institut f\" ur Angewandte Mathematik \\ Welfengarten 1 \\ D--30167 Hannover\\ Germany}
\email{walker@ifam.uni-hannover.de}

\begin{document}

\date{\today}

\maketitle

\begin{abstract}
Touchdown is shown to be the only possible finite time singularity that may take place in a free boundary problem modeling a three-dimensional microelectromechanical system. The proof relies on the energy structure of the problem and uses smoothing effects of the semigroup  generated in $L_1$ by the bi-Laplacian with clamped boundary conditions.
\end{abstract}

\section{Introduction}

We consider a model for a three-dimensional microelectromechanical system (MEMS) including two components, a rigid ground plate of shape $D\subset \mathbb{R}^2$ and an elastic plate of the same shape (at rest) which is suspended above the rigid one and clamped on its boundary, see Figure~\ref{fig1}.  Both plates being conducting, holding them at different voltages generates a Coulomb force across the device. This, in turn, induces a deformation of the elastic plate, thereby modifying the geometry of the device and transforming electrostatic energy into mechanical energy. When applying a sufficiently large voltage difference, a well-known phenomenon that might occur is that the two plates come into contact; that is, the elastic plate touches down on the rigid plate. For this feature~--~usually referred to as \textit{pull-in instability} or \textit{touchdown} \cite{FMCCS05, PeBe03} -- some mathematical models have been developed recently \cite{BGP00, FMCCS05, LW18, PeBe03, Pel01a}.  Since the pioneering works \cite{BGP00, FMPS06, GPW05, Pel01a}, their mathematical analysis has been the subject of numerous papers. We refer to \cite{EGG10, LW17} for a more complete account and an extensive list of references.

\begin{figure}\label{fig1}
	\begin{tikzpicture}
	\draw[blue, line width = 2pt] plot[domain=-5:-1] (\x,{-1-cos((pi*(\x+1)/4) r)});
	\draw[blue, line width = 2pt] plot[domain=-1:1] (\x,{-2.5+ 0.5*cos((pi*(\x+1)/2) r)});
	\draw[blue, line width = 2pt] plot[domain=1:5] (\x,{-1.5+1.5*sin((pi*(\x-3)/4) r)});
	\draw[black, line width = 1.5pt] (-5,0)--(-5,-5);
	\draw[black, line width = 3pt] (-5,-5)--(5,-5);
	\draw[black, line width = 1.5pt] (5,-5)--(5,0);
	\draw[black, line width = 1pt, dashed] (-6,0)--(0,0);
	\draw[black, line width = 1pt, dashed, |->] (0,0)--(6,0);
	\node at (5.8,0.3) {$x_1$};
	\node at (5,0.3) {$1$};
	\node at (-5.1,0.3) {$-1$};
	\draw[black, line width = 0.5pt, dashed, ->] (0,0)--(1,0.75);
	\node at (1,0.45) {{\small $x_2$}};
	\draw[black, line width = 1pt, arrows=<-] (-1.4,-1.95)--(-1.4,0);
	\node at (-1.1,-0.9) {$u$};
	\node at (-2,-3.5) {$\Omega(u)$};
	\draw (3.5,-5.75) edge[->, bend left, line width = 1pt] (2,-5.075);
	\node at (4.85,-5.75) {\textbf{ground plate} $D$};
	\draw ((6,-1) edge[->,bend left, line width = 1pt] (4,-0.45);
	\node at (7.05,-1) {\textbf{elastic plate}};
	\draw[black, line width = 1pt, dashed, arrows = |->] (0,0)--(0,1);
	\node at (-0.3,0.25) {$0$};
	\draw[black, line width = 1pt, dashed] (0,-5)--(0,0);
	\node at (-0.3,0.9) {$z$};
	\end{tikzpicture}
	\caption{Cross section of an idealized MEMS device}
\end{figure}
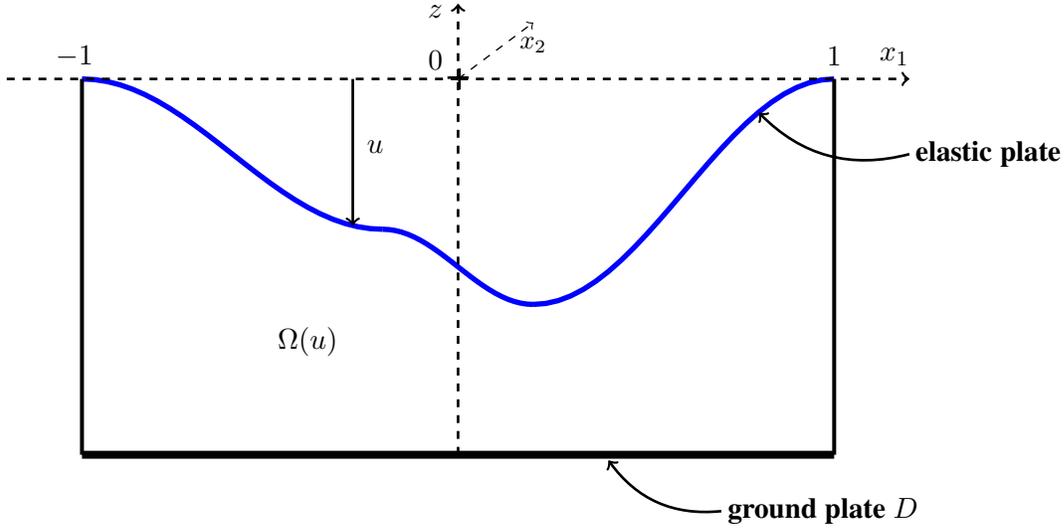

We focus here on a model describing the evolution of the vertical deformation of the elastic plate from rest and the electrostatic potential between the plates. More precisely, we assume that $D$ is a bounded and convex domain in $\R^2$  with a $C^\infty$-smooth boundary. Then, after an appropriate rescaling and neglecting inertial forces, the ground plate is located at $z=-1$ while the elastic plate's rest position is at $z=0$, and the evolution of the vertical deformation $u=u(t,x)$ of the elastic plate at time $t>0$ and position $x\in D$ is given by 
\begin{subequations}\label{u}
\begin{equation}\label{u1}
\partial_t u+\beta\Delta^2 u-\big(\tau+a\|\nabla u\|_{L_2(D)}^2\big)\Delta u = - \lambda\, g(u) \ ,\quad x\in D\ ,\quad t>0\,,
\end{equation}
where
\begin{equation}\label{u2}
g(u(t))(x) := \varepsilon^2 |\nabla\psi_{u(t)}(x,u(t,x))|^2 + |\partial_z \psi_{u(t)}(x,u(t,x))|^2 \ ,\quad x\in D\ ,\quad t>0\,.
\end{equation}
Throughout the paper, $\nabla$ and $\Delta$ denote the gradient and the Laplace operator with respect to $x\in D$, respectively. 
We supplement \eqref{u1} with clamped boundary conditions
\begin{equation}\label{bcu}
u=\partial_\nu u=0\,  ,\quad x\in \partial D\ , \quad t>0\,,
\end{equation}
and initial condition
\begin{equation}\label{ic}
u(0,x)=u^0(x)\ ,\quad x\in D\ .
\end{equation}
\end{subequations}
As for the electrostatic potential $\psi_{u(t)}(x,z)$, it is defined for $t>0$ and $(x,z)\in \Omega(u(t))$, where $\Omega(u(t))$ is the three-dimensional cylinder
$$
\Omega(u(t)) := \left\{ (x,z)\in D\times (-1,\infty)\ :\ -1 < z < u(t,x) \right\}
$$
 enclosed within the rigid ground plate at $z=-1$ and the deflected elastic plate at $z=u(t)$.
For each time $t>0$, the electrostatic potential $\psi_{u(t)}$ solves the rescaled Laplace equation
\begin{subequations}\label{psi0}
\begin{equation}\label{psi}
\varepsilon^2\Delta\psi_{u(t)} + \partial_z^2\psi_{u(t)} =0\ ,\quad (x,z)\in \Omega(u(t))\ ,\quad t>0\ ,
\end{equation}
supplemented with non-homogeneous Dirichlet boundary conditions
\begin{equation}\label{bcpsi}
\psi_{u(t)}(x,z)=\frac{1+z}{1+u(t,x)}\ ,\quad (x,z)\in  \partial\Omega(u(t))\,, \quad t>0 \  .
\end{equation}
\end{subequations}

In \eqref{u}-\eqref{psi0}, the aspect ratio $\varepsilon>0$ is the ratio between vertical and horizontal dimensions of the device while $\lambda>0$ is proportional to the square of the applied voltage difference. The parameters $\beta>0$, $\tau\ge 0$, and $a\ge 0$ result from the modeling of the mechanical forces and are related to bending and stretching of the elastic plate, respectively. We emphasize that \eqref{u}-\eqref{psi0} is a nonlinear and nonlocal system of partial differential equations featuring a time-varying boundary, which makes its analysis rather involved. Still, its local in time well-posedness can be shown in a suitable functional setting, as we recall below, and the aim of this note is to improve the criterion for global existence derived in \cite{LW16a}.

\section{Main Result}

Expanding upon the above discussion on global existence we recall the following result established in \cite[Theorem~1.1]{LW16a}.

\begin{theorem}\label{A}
Let $4\xi\in (7/3,4)$, and consider an initial value $u^0\in  W_{2}^{4\xi}(D)$ such that $u^0>-1$ in $D$ and $u^0=\partial_\nu u^0=0$ on $\partial D$. 

\begin{itemize}

\item[(i)] There is a unique solution $u$ to \eqref{u} on the maximal interval of existence $[0,T_m)$ in the sense that
\begin{equation}\label{reg}
u\in C\big([0,T_m), W_{2}^{4\xi}(D)\big) \cap C\big((0,T_m), W_{2}^4(D)\big) \cap C^1\big((0,T_m),L_2(D)\big) 
\end{equation}
satisfies \eqref{u} together with
$$
u(t,x)>-1\ ,\quad (t,x)\in [0,T_m)\times D\ , 
$$ 
and $\psi_{u(t)}\in W_2^2\big(\Omega(u(t))\big)$ solves \eqref{psi0} in $\Omega(u(t))$ for each $t\in [0,T_m)$. 

\item[(ii)] If $T_m<\infty$, then
\begin{equation}\label{U2}
\lim_{t\to T_m} \|u(t)\|_{W_2^{4\xi}(D)} = \infty \qquad\text{ or }\qquad \lim_{t\to T_m} \min_{x\in \bar{D}} u(t,x) = - 1\,.
\end{equation}
\end{itemize}
\end{theorem}

It is worth pointing out that, since $\Omega(u(t))$ is only a Lipschitz domain, the $W_2^2$-regularity of $\psi_{u(t)}$ does not seem to follow from standard elliptic theory. Actually, this property is one of the cornerstones in the proof of Theorem~\ref{A} and guarantees that the function $g$ in \eqref{u2} is well-defined (see Proposition~\ref{P2.1} below).

Further results regarding \eqref{u}-\eqref{psi0} are to be found in  \cite{LW16a}. In particular, global existence holds true under additional smallness assumptions on both $\lambda$ and $u^0$. Moreover, stationary solutions exist for small values of $\lambda$ and, when $D$ is a ball in $\mathbb{R}^2$, no stationary solution exists for $\lambda$ large enough. This last property is actually connected with the touchdown phenomenon already alluded to in the introduction. In the same vein, whether a finite time singularity may occur for the evolution problem for suitable choices of $\lambda$ and $u^0$ is yet an open problem, though such a feature is expected on physical grounds.

Coming back to the global existence issue, the criterion \eqref{U2} stated in Theorem~\ref{A} entails that non-global solutions blow up in finite time in the Sobolev space $W_2^{4\xi}(D)$ or a finite time touchdown of the elastic plate on the ground plate occurs,
the occurrence of both simultaneously being not excluded \textit{a priori}. 
From a physical point of view, however, only the latter seems possible. For the investigation of the dynamics of MEMS devices it is thus of great importance to rule out mathematically the norm blowup in finite time. In \cite{LW14a} this was done if $D=(-1,1)$ is one-dimensional, that is, in case the elastic part is a beam or a rectangular plate that is homogeneous in one direction. The situation considered herein, where $D$ is an arbitrary two-dimensional (convex) domain, is more delicate. Indeed, the right-hand side of \eqref{u} --~being given by the square of the gradient trace of the electrostatic potential~-- has much less regularity properties due to the fact that the moving boundary problem \eqref{psi0} for the electrostatic potential is posed in a three-dimensional domain $\Omega(u)$. We shall see, however, that we can overcome this difficulty using the gradient flow structure of the evolution problem along with the regularizing effects of the fourth-order operator in (negative) Besov spaces. More precisely, we  shall show the following result.

\begin{theorem}\label{B}
	Under the assumptions of Theorem~\ref{A} let $u$ be the unique maximal solution to \eqref{u} on the maximal interval of existence $[0,T_m)$. Assume that there are $T_0>0$ and $\kappa_0\in (0,1)$ such that 
	\begin{equation}
	u(t)\ge -1 + \kappa_0\ \ \text{in}\ \ D\,, \quad t\in [0,T_m)\cap [0,T_0]\,. \label{td0}
	\end{equation}
	Then $T_m\ge T_0$.
	
	Moreover, if, for each $T>0$, there is $\kappa(T)\in (0,1)$ such that 
	\begin{equation*}
	u(t)\ge -1+\kappa(T)\ \ \text{in}\ \ D\,,\quad t\in [0,T_m)\cap [0,T]\,,
	\end{equation*}
	then $T_m=\infty$.
\end{theorem}

The second statement in Theorem~\ref{B} obviously follows from the first one applied to an arbitrary $T_0>0$. The proof of Theorem~\ref{B} is given in the next section. As mentioned above and similarly to the case $D=(-1,1)$ considered in \cite{LW14a}, it relies on the gradient flow structure of \eqref{u}-\eqref{psi0}, where the corresponding energy is given by
$$
\mathcal{E}(u):=\mathcal{E}_m(u)-\lambda\mathcal{E}_e(u)
$$
with mechanical energy
$$
\mathcal{E}_m(u):=\frac{\beta}{2}\|\Delta u\|_{L_2(D)}^2+\frac{\tau}{2}\|\nabla u\|_{L_2(D)}^2 +\frac{a}{4}\|\nabla u\|_{L_2(D)}^4
$$
and electrostatic energy
$$
\mathcal{E}_e(u):=\int_{\Omega(u)} \left( \varepsilon^2 |\nabla\psi_u(x,z)|^2 + |\partial_z \psi_u(x,z)|^2 \right)\,\rd (x,z)\,.
$$
We shall see that assuming the lower bound \eqref{td0} on the solution $u$ provides a control on the electrostatic energy. Using the gradient flow structure we thus derive a bound on the (a priori unbounded) mechanical energy and, in turn, on the $W_2^2(D)$-norm of $u(t)$ for $t\in [0,T_m)\cap [0,T_0]$. This yields  an $L_1(D)$-bound on the right-hand side of \eqref{u}. We then apply semigroup techniques in negative Besov spaces to obtain a bound on $u(t)$ in the desired Sobolev norm of $W_2^{4\xi}(D)$ for $t\in [0,T_m)\cap [0,T_0]$ which only depends on $T_0$ and $\kappa_0$.

\begin{remark}\label{rem2ndorder}
	It is worth pointing out that the issue whether a norm blowup or touchdown occurs in finite time is still an open problem for the second-order case $\beta=0$ (and $\tau>0$), even in the one-dimensional setting $D=(-1,1)$.
\end{remark}

\section{Proof of Theorem~\ref{B}}
%
\newcounter{NumConst}

Suppose the assumptions of Theorem~\ref{A} and let $u$ denote the unique maximal solution to \eqref{u} on the maximal interval of existence $[0,T_m)$. We want to show that, if  \eqref{td0} is satisfied, then
\begin{equation*}
\| u(t)\|_{W_2^{4\xi}(D)}\le c(T_0,\kappa_0)\,, \quad t\in [0,T_m)\cap [0,T_0]\,,
\end{equation*}
so that Theorem~\ref{A}~(ii) in turn implies Theorem~\ref{B}. To this end we first need to derive suitable estimates on the right-hand side $g(u)$ of \eqref{u} given by the square of the gradient trace of the electrostatic potential $\psi_u$.
 
\subsection{Estimates on the electrostatic potential}

In the following we let $\kappa\in (0,1)$ and set 
$$
\mathcal{S}(\kappa):=\{v\in W_3^2(D)\,:\, v=0 \text{ on }\partial D \text{ and } v\ge -1+\kappa \text{ in } D\}
\,.
$$
We begin with the regularity of the variational solution to \eqref{psi0}, see \cite{LW16a}. 

\begin{prop}\label{P2.1}
Given $v\in \mathcal{S}(\kappa)$, there is a unique solution $\psi_v\in W_2^2(\Omega(v))$   to
\begin{subequations}\label{psiv0}
\begin{align}
\varepsilon^2\Delta\psi +\partial_z^2\psi =0\ ,\quad (x,z)\in \Omega(v)\, ,\label{psiv}\\
\psi(x,z)=\frac{1+z}{1+v(x)}\ ,\quad (x,z)\in  \partial\Omega(v)\,, \label{bcpsiv}
\end{align}
\end{subequations}
in the cylinder
$$
\Omega(v) := \left\{ (x,z)\in D\times (-1,\infty)\ :\ -1 < z < v(x) \right\}\,.
$$
Furthermore, $g(v)\in L_2(D)$.
\end{prop}

We recall that the $L_2(D)$-integrability of $g(v)$ is a straightforward consequence of  $\psi_v\in W_2^2(\Omega(v))$. Indeed the latter implies that $\big(x\mapsto \nabla \psi_v(x,v(x))\big)\in W_2^{1/2}(D)\hookrightarrow L_4(D)$.\medskip

We next provide pointwise estimates on $\psi_v$.

\begin{lem}\label{L1.1}
Let $v\in \mathcal{S}(\kappa)$. Then, for $(x,z)\in \Omega(v)$, 
$$
0\le \psi_v(x,z)\le \min\left\{1,\frac{1+z}{\kappa}\right\}\,.
$$
\end{lem}

\begin{proof}
Clearly, $(x,z)\mapsto m$ is a solution to \eqref{psiv} for $m=0,1$ and $0\le \psi_v\le 1$ on $\partial \Omega(v)$ since $v=0$ on $\partial D$, hence $0\le \psi_v \le 1$ in $\Omega(v)$ by the comparison principle. Moreover, setting $\Sigma(x,z):=(1+z)/\kappa$ for $(x,z)\in \Omega(v)$, it readily follows that $\Sigma$ is a supersolution to \eqref{psiv0} so that $\psi_v\le \Sigma$ in $\Omega(v)$ again by the comparison principle.
\end{proof}

Lemma~\ref{L1.1} provides uniform estimates on the derivatives of $\psi_v$ on the $v$-independent part of the boundary of $\Omega(v)$.

\begin{cor}\label{C2}
Let $v\in \mathcal{S}(\kappa)$. If $x\in \partial D$ and $z\in (-1,0)$, then $\partial_z\psi_v(x,z)=1$, while if $x\in D$, then
$$
0\le \partial_z\psi_v(x,-1)\le \frac{1}{\kappa}\,,\qquad \partial_z\psi_v(x,v(x))\ge 0\,,
$$
and
$$
\nabla\psi_v(x,-1)=0\,,\qquad \nabla\psi_v(x,v(x))=-\partial_z\psi_v(x,v(x))\nabla v(x)\,.
$$
\end{cor}

\begin{proof}
The first assertion follows from $\psi_v(x,z)=1+z$, $(x,z)\in \partial D\times (-1,0)$. Next, from \eqref{bcpsiv} and Lemma~\ref{L1.1} we derive, for $(x,z)\in D\times (-1,0)$,
$$
0\le \frac{\psi_v(x,z)-\psi_v(x,-1)}{1+z}\le \frac{1}{\kappa}\,,\qquad \psi_v(x,v(x))-\psi_v(x,z) \ge 0\,,
$$
hence
$$
0\le \partial_z\psi_v(x,-1)\le \frac{1}{\kappa}\,,\qquad \partial_z\psi_v(x,v(x))\ge 0\,.
$$
The formulas for $\nabla\psi_v$ follow immediately from $\psi_v(x,-1)=0$ and $\psi_v(x,v(x))=1$ for $x\in D$ due to \eqref{bcpsiv}.
\end{proof}

Given $v\in \mathcal{S}(\kappa)$ we next introduce the notation
\bqn\label{g}
\gamma(x):=\partial_z\psi_v(x,v(x))\,,\qquad \gamma_b(x):=\partial_z\psi_v(x,-1)
\eqn
for $x\in D$ and recall the following identity, which is proven in \cite[Lemma~5]{ELW13} in the one-dimensional case $D=(-1,1)$,

\begin{lem}\label{L3}
Let $v\in \mathcal{S}(\kappa)$. Then, with the notation \eqref{g},
\begin{equation*}
\int_D \left( 1 + \varepsilon^2 \vert\nabla v\vert^2\right) \left(\gamma^2-2\gamma\right)\, \mathrm{d}x =  \int_D\left( \gamma_b^2-2\gamma_b \right) \,\mathrm{d}x \, . \label{b5}
\end{equation*}
\end{lem}

\begin{proof}
We recall the proof for the sake of completeness and point out that it is somewhat related to the Rellich equality \cite[Equation~(5.2)]{Necas67}. We multiply the rescaled Laplace equation~\eqref{psiv} by $\partial_z\psi_v-1$ and integrate over $\Omega(v)$. Denoting the outward unit normal vector field to $\partial D$ and the surface measure on $\partial D$ by $\nu$ and $\sigma$, respectively, we deduce from Green's formula that
\begin{align*}
0 = & \int_{\Omega(v)}\left(\ve^2\Delta\psi_v +\partial_z^2\psi_v\right) \left(\partial_z\psi_v-1\right)\,\rd (x,z)\\
= & \ \ve^2\int_{\partial D} \int_{-1}^0 \left(\partial_z\psi_v-1\right)\nabla\psi_v\cdot \nu\,\rd z\,\rd \sigma\\
&-\ve^2\int_{ D}  \left(\partial_z\psi_v(x,v(x))-1\right) \nabla\psi_v(x,v(x))\cdot \nabla v(x)\,\rd x\\
& - \ve^2\int_{\Omega(v)}  \nabla\psi_v\cdot \partial_z\nabla\psi_v\, \rd (x,z)+\int_D \left(\frac{(\partial_z\psi_v(x,v(x)))^2}{2}-\partial_z\psi_v(x,v(x))\right)\,\rd x\\
& - \int_D \left(\frac{(\partial_z\psi_v(x,-1))^2}{2}-\partial_z\psi_v(x,-1)\right) \,\rd x\,.
\end{align*}
Due to Corollary~\ref{C2} the first integral on the right-hand side vanishes while the others can be simplified to get
\begin{align*}
0 = &  -\ve^2\int_D (\gamma(x)-1)\nabla\psi(x,v(x))\cdot \nabla v(x)\,\rd x-\frac{\ve^2}{2}\int_D\left\vert\nabla\psi_v(x,v(x))\right\vert^2\,\rd x\\
&+\frac{\ve^2}{2}\int_D\left\vert\nabla\psi_v(x,-1)\right\vert^2\,\rd x +\int_D\left(\frac{\gamma^2}{2}-\gamma\right)\, \rd x- \int_D\left(\frac{\gamma_b^2}{2}-\gamma_b\right)\, \rd x\\
=&\ \ve^2\int_D (\gamma-1)\gamma \vert \nabla v\vert^2\,\rd x-\frac{\ve^2}{2}\int_D \gamma^2\vert\nabla v\vert^2\,\rd x+ \int_D\left(\frac{\gamma^2}{2}-\gamma\right)\, \rd x- \int_D\left(\frac{\gamma_b^2}{2}-\gamma_b\right)\, \rd x\\
=&\ \int_D \left(\frac{\gamma^2}{2}-\gamma\right) \left(1+\ve^2\vert \nabla v\vert^2\right)\,\rd x- \int_D\left(\frac{\gamma_b^2}{2}-\gamma_b\right)\, \rd x\,,
\end{align*}
which yields the assertion.
\end{proof}

Given $v\in \mathcal{S}(\kappa)$ we recall that
$$
g(v)(x):= \varepsilon^2 |\nabla\psi_v(x,v(x))|^2 + |\partial_z \psi_v(x,v(x))|^2\,,\quad x\in D\,,
$$
with $\psi_v$ still denoting the solution to \eqref{psiv0}. The next result bounds the $L_1(D)$-norm of $g(v)$ in terms of the $H^1(D)$-norm of $v$.

\begin{cor}\label{C4}
For $v\in \mathcal{S}(\kappa)$,
$$
\| g(v)\|_{L_1(D)}\le \left(4+\frac{2}{\kappa^2}\right) \vert D\vert +4 \ve^2\|\nabla v\|_{L_2(D)}^2\,.
$$
\end{cor}

\begin{proof}
Since
$$
g(v)(x)= \ve^2|\nabla\psi_v(x,v(x))|^2+|\partial_z\psi_v(x,v(x))|^2=\left(1+\ve^2\vert\nabla v(x)\vert^2\right) \gamma(x)^2
$$
for $x\in D$ by Corollary~\ref{C2}, we deduce from Corollary~\ref{C2} and Lemma~\ref{L3} that
\begin{equation*}
\begin{split}
\| g(v)\|_{L_1(D)} &= 2\int_D \left(1+\ve^2\vert\nabla v\vert^2\right) \gamma\,\rd x+\int_D (\gamma_b^2-2\gamma_b)\,\rd x\\
&\le \frac{1}{2}\int_D \left(1+\ve^2\vert\nabla v\vert^2\right) \gamma^2\,\rd x +2\int_D \left(1+\ve^2\vert\nabla v\vert^2\right)\,\rd x + \frac{\vert D\vert}{\kappa^2} \\
&\le  \frac{1}{2} \| g(v)\|_{L_1(D)}+ 2\ve^2 \|\nabla v\|_{L_2(D)}^2+ \left(2+\frac{1}{\kappa^2}\right)\vert D\vert \,,
\end{split}
\end{equation*}
from which the assertion follows.
\end{proof}

We next recall the following identity for the electrostatic energy established in \cite[Equation~(3.13)]{LW16b} in the one-dimensional case $D=(-1,1)$. We extend it here to the two-dimensional setting, also providing a simpler proof below.

\begin{lem}\label{L5}
For $v\in \mathcal{S}(\kappa)$,
$$
\mathcal{E}_e(v) =\vert D\vert -\int_D v \left(1+\ve^2\vert\nabla v\vert^2\right) \gamma\,\rd x\,.
$$
\end{lem}

\begin{proof}
We multiply the rescaled Laplace equation \eqref{psiv} by $\psi_v(x,z)-1-z$ and integrate over $\Omega(v)$. As in the proof of Lemma~\ref{L3} we use Green's formula to obtain
\begin{align*}
0 = & \int_{\Omega(v)}\left(\ve^2\Delta\psi_v +\partial_z^2\psi_v\right)(x,z) \left(\psi_v(x,z)-1-z\right)\,\rd (x,z)\\
= & \ \ve^2\int_{\partial D} \int_{-1}^0 \left(\psi_v(x,z)-1-z\right)\nabla\psi_v\cdot \nu\,\rd z\,\rd \sigma\\
&-\ve^2\int_{ D}  \left(\psi_v(x,v(x))-1-v(x)\right)\nabla\psi_v(x,v(x))\cdot \nabla v(x)\,\rd x\\
& - \ve^2\int_{\Omega(v)}  \vert\nabla\psi_v\vert^2\, \rd (x,z)+\int_D \left(\psi_v(x,v(x))-1-v(x)\right)\partial_z\psi_v (x,v(x))\,\rd x\\
& - \int_D \psi_v(x,-1)\partial_z\psi_v (x,-1)\,\rd x-\int_{\Omega(v)}\left(\partial_z\psi_v-1\right)\partial_z\psi_v\,\rd (x,z)\,.
\end{align*}
Employing \eqref{bcpsi} we see that the first and the fifth term on the right-hand side vanish while the others can be gathered due to Corollary~\ref{C2} as
\begin{align*}
0 = &-\ve^2\int_{D}  v\vert \nabla v\vert^2 \gamma \,\rd x - \ve^2\int_{\Omega(v)}  \vert\nabla\psi_v\vert^2\, \rd (x,z) -\int_D v \gamma\,\rd x\\
&-\int_{\Omega(v)}\left(\partial_z\psi_v\right)^2\,\rd (x,z) +\int_D\left(\psi_v(x,v(x))-\psi_v(x,-1)\right)\,\rd x\,.
\end{align*}
The last integral being equal to $\vert D\vert$ according to \eqref{bcpsi}, we obtain 
$$
\mathcal{E}_e(v)= \vert D\vert-\int_Dv \left(1+\ve^2\vert\nabla v\vert^2\right)\gamma\,\rd x\,,
$$
hence the assertion.
\end{proof}

We are now in a position to derive a lower bound on the total energy.

\begin{cor}\label{C6}
For $v \in \mathcal{S}(\kappa)$, 
$$
\mathcal{E}(v)\ge \mathcal{E}_m(v)- 3\lambda\ve^2\|\nabla v\|_{L_2(D)}^2- \lambda |D| \left( 4 + \frac{1}{2\kappa^2} \right) \,.
$$
\end{cor}

\begin{proof}
Since $v\ge -1$ in $D$ and, by Corollary~\ref{C2}, $\gamma\ge 0$ in $D$, we infer from Lemma~\ref{L5} that
\begin{equation*}
\begin{split}
\mathcal{E}(v) &= \mathcal{E}_m(v)-\lambda \mathcal{E}_e(v)=\mathcal{E}_m(v) -\lambda \vert D\vert +\lambda \int_Dv \left(1+\ve^2\vert\nabla v \vert^2\right)\gamma\,\rd x
\\
&\ge \mathcal{E}_m(v) -\lambda \vert D\vert -\lambda \int_D \left(1+\ve^2\vert\nabla v\vert^2\right)\gamma\,\rd x\,,
\end{split}
\end{equation*}
so that the Cauchy-Schwarz inequality and Corollary~\ref{C4} imply that
\begin{equation*}
\begin{split}
\mathcal{E}(v) 
&\ge \mathcal{E}_m(v) -\lambda \vert D\vert -\lambda \left(\int_D \left(1+\ve^2\vert\nabla v\vert^2\right)\,\rd x\right)^{1/2} \|g(v)\|_{L_1(D)}^{1/2}\\
&\ge  \mathcal{E}_m(v) -\lambda \vert D\vert \\
& \quad -\lambda \left(\int_D \left(1+\ve^2\vert\nabla v\vert^2\right)\,\rd x\right)^{1/2} \left( \frac{2 \vert D\vert}{\kappa^2}+ 4 \int_D\left(1+\ve^2\vert\nabla v\vert^2\right)\,\rd x \right)^{1/2}\\
&\ge  \mathcal{E}_m(v) -\lambda \vert D\vert -\frac{\sqrt{2|D|}\lambda}{\kappa} \left(\int_D \left(1+\ve^2\vert\nabla v\vert^2\right)\,\rd x\right)^{1/2} \\
& \qquad -2\lambda \int_D \left(1+\ve^2\vert\nabla v\vert^2\right)\,\rd x\,.
\end{split}
\end{equation*}
The assertion follows then from Young's inequality.
\end{proof}

\subsection{Estimates on the plate deflection}

Under the assumptions of Theorem~\ref{A} let now $u$ be the unique maximal solution to \eqref{u} on the maximal interval of existence $[0,T_m)$.
We may assume that $7/3<4\xi<3$. Let $\kappa_0\in (0,1)$ and $T_0>0$ be such that \eqref{td0} holds true; that is,
\begin{equation}\label{1}
u(t,x)\ge -1+\kappa_0\,,\quad t\in [0,T_m)\cap [0,T_0]\,, \ x\in D\,.
\end{equation}
Throughout this section, $c$ denotes a positive constant which may vary from line to line and depends only on $\beta$, $\tau$, $a$, $\lambda$, $D$, $\varepsilon$, $u^0$, $\kappa_0$, and $T_0$ (in particular, it does not depend on $T_m$).

To prove Theorem~\ref{B} we shall show that 
\begin{equation}\label{1a}
\| u(t)\|_{W_2^{4\xi}(D)}\le c\,, \quad t\in [0,T_m)\cap [0,T_0]\,,
\end{equation}
the assertion then follows from Theorem~\ref{A}~(ii). Note that \eqref{1} just means that $u(t)\in \mathcal{S}(\kappa_0)$ for $t\in [0,T_m)\cap [0,T_0]$ so that the results of the preceding section apply (with $\kappa=\kappa_0$).

\medskip

We first provide an $L_2(D)$-bound on $u$.
 
\begin{lem}\label{L2.1}
There is $c>0$ such that
$$
\| u(t)\|_{L_2(D)}\le c\,,\quad t\in [0,T_m)\cap [0,T_0]\,.
$$
\end{lem}

\begin{proof}
Let $t\in [0,T_m)$. It readily follows from \eqref{u} and the lower bounds $u(t)\ge -1$ and $g(u(t))\ge 0$ in $D$ that
$$
\frac{1}{2}\frac{\rd}{\rd t}\| u(t)\|_{L_2(D)}^2+2\mathcal{E}_m(u(t))=-\lambda\int_D u(t) g(u(t))\,\rd x\le \lambda \|g(u(t))\|_{L_1(D)}\,.
$$
Now, Corollary~\ref{C4} along with  interpolation and Young's inequality imply, for $t\in [0,T_m)\cap [0,T_0]$,
\begin{align*}
\| g(u(t))\|_{L_1(D)}&\le c\left (1 + \|\nabla u(t)\|_{L_2(D)}^2 \right)\\
& \le c\left (1 + \|u(t)\|_{L_2(D)}\, \|\Delta u(t)\|_{L_2(D)}\right)\\
&\le \frac{1}{\lambda} \mathcal{E}_m(u(t)) + c \left(1+\|u(t)\|_{L_2(D)}^2\right) \,.
\end{align*}
Combining the two inequalities yields
$$
\frac{1}{2}\frac{\rd}{\rd t}\| u(t)\|_{L_2(D)}^2+\mathcal{E}_m(u(t))\le c \left(1+\|u(t)\|_{L_2(D)}^2\right)\,,\quad t\in [0,T_m)\cap [0,T_0]\,,
$$
from which the assertion follows.
\end{proof}

We next show that the lower bound \eqref{1} on $u$ implies that the mechanical energy is dominated by the total energy.

\begin{lem}\label{L2.2}
There is $c>0$ such that
$$
\mathcal{E}(u(t))\ge \frac{1}{2}\mathcal{E}_m(u(t))-c\,,\quad t\in [0,T_m)\cap [0,T_0]\,.
$$
\end{lem}

\begin{proof}
We infer from Corollary~\ref{C6} along with interpolation and Young's inequality that, for some constant $c>0$,
\begin{equation*}
\begin{split}
\mathcal{E}(u(t))&\ge \mathcal{E}_m(u(t))- c\|u(t)\|_{L_2(D)} \mathcal{E}_m(u(t))^{1/2}-c\\
&\ge \frac{1}{2} \mathcal{E}_m(u(t)) -c \left(1+\|u(t)\|_{L_2(D)}^2\right)
\end{split}
\end{equation*}
for $t\in [0,T_m)\cap [0,T_0]$. Lemma~\ref{L2.1} yields the claim.
\end{proof}

We next exploit the gradient flow structure of the evolution problem to obtain additional estimates.

\begin{cor}\label{C2.3}
There is $c>0$ such that
$$
\| u(t)\|_{H^2(D)}+\int_0^t \|\partial_t u(s)\|_{L_2(D)}^2\,\rd s\le c\,,\quad t\in [0,T_m)\cap [0,T_0]\,.
$$
\end{cor}

\begin{proof}
Analogously to \cite[Proposition~1.3]{LW14a} (see also \cite{LW16b}) the energy inequality
$$
\mathcal{E}(u(t))+\int_0^t\|\partial_t u(s)\|_{L_2(D)}^2\,\rd s \le \mathcal{E}(u^0)\,,\quad t\in [0,T_m)\,,
$$
holds; that is, due to Lemma~\ref{L2.2},
$$
\mathcal{E}(u^0)\ge \frac{1}{2}\mathcal{E}_m(u(t))-c+\int_0^t\|\partial_t u(s)\|_{L_2(D)}^2\,\rd s\,,\quad t\in [0,T_m)\cap [0,T_0]\,.
$$
The claim follows then from the fact that $\mathcal{E}(u^0)<\infty$ and the definition of $\mathcal{E}_m$.
\end{proof}

Combining now Corollary~\ref{C4} and Corollary~\ref{C2.3} we readily obtain an $L_1(D)$-bound on the right-hand side of \eqref{u}.

\begin{cor}\label{C2.4}
There is $c>0$ such that
$$
\|g(u(t))\|_{L_1(D)}\le c\,,\quad t\in [0,T_m)\cap [0,T_0]\,.
$$
\end{cor}

\subsection{Proof of Theorem~\ref{B}}

It remains to prove that the $L_1(D)$-bound from Corollary~\ref{C2.4} implies a bound on $u$ in the Sobolev space $W_2^{4\xi}(D)$, that is, inequality \eqref{1a}.

For this purpose we introduce $B_{1,1,\mathcal{B}}^{s}(D)$ for $s\in \R\setminus\{1,2\}$, i.e.  the Besov space $B_{1,1}^{s}(D)$ incorporating the boundary conditions appearing in \eqref{bcu} (if meaningful):
$$
B_{1,1,\mathcal{B}}^{s}(D):=\left\{\begin{array}{ll}
\{w\in  B_{1,1}^{s}(D)\ :\  w=\partial_\nu w=0 \text{  on } \partial D\}\ , & s > 2\,,\\
& \\
\{w\in  B_{1,1}^{s}(D)\ :\ w=0 \text{  on } \partial D\}\ , &s \in (1,2)\,,\\
&\\
B_{1,1}^{s}(D)\ , &s< 1\,.\end{array}\right.
$$
The spaces $W_{2,\mathcal{B}}^{s}(D)$ are defined analogously with $B_{1,1}^{s}$ replaced by $W_{2}^{s}$, but for $s>3/2$, $s\in (1/2,3/2)$, and $s<1/2$, respectively. 

From now on, we fix $\alpha\in (4\xi-3,0)$.  Hereafter, the constant $c$ may also depend on $\xi$ and $\alpha$ (but still not on $T_m$). The dependence upon additional parameters is indicated explicitly.

\begin{lem}\label{SG}
The operator $-A$, given by
$$
-A v:= (-\beta \Delta^2 +\tau\Delta) v \,,\quad v\in B_{1,1,\mathcal{B}}^{4+\alpha}(D)\,,
$$
generates an analytic semigroup $\{e^{-tA}\,;\, t\ge 0\}$ on $B_{1,1}^{\alpha}(D)$ and, when restricted to $W_{2,\mathcal{B}}^{4}(D)$, on $L_2(D)$. Given $\theta\in (0,1)$ with $\theta\not\in\{(1-\alpha)/4,(2-\alpha)/4\}$, there are $c>0$ and $c(\theta)>0$ such that, for $t\in [0,T_0]$,
\bqn\label{6}
\|e^{-tA}\|_{\mathcal{L}(W_{2,\mathcal{B}}^{4\xi}(D))} \le c \qquad \ \text{ and}\qquad  \ t^\theta \|e^{-tA}\|_{\mathcal{L}(B_{1,1}^{\alpha}(D),B_{1,1,\mathcal{B}}^{4\theta+\alpha}(D))} \le c(\theta)\,.
\eqn
\end{lem}

\begin{proof}
It is readily seen that the principal part $-\beta\Delta^2$ of the operator $-A$ with symbol $-A_0(i\zeta)=-\beta\vert \zeta\vert^4$ is elliptic and, when supplemented with the normal system  $B=(\mathrm{tr},\partial_\nu)$ of boundary operators, satisfies the Lopatinskii-Shapiro condition~(o) from \cite[p.~268]{Gu_MN}. Indeed, given $x\in\partial D$, $\zeta\in\R^2$, $r\ge 0$, and $\vartheta\in[-\pi/2,\pi/2]$ with $\zeta\cdot\nu(x)=0$ and $(\zeta,r)\not=(0,0)$, this condition requires that zero is the only bounded solution on $[0,\infty)$ to 
$$
\left(-A_0\big(i\zeta+\nu(x)\partial_t\big)-re^{i\vartheta}\right)v=0\,,\qquad B\big(i\zeta+\nu(x)\partial_t\big)v(0)=0\,,
$$ 
that is, to 
\begin{subequations}\label{r}
\begin{align}
\partial_t^4 v(t) -2\vert\zeta\vert^2\partial_t^2 v(t) +\left(\vert\zeta\vert^4+re^{i\vartheta}\right) v(t)&=0\,,\quad t>0\,,\,\\
v(0)= \partial_t v(0)&=0\,.
\end{align}
\end{subequations}
Introducing 
$$
M_\pm:=\sqrt{\vert\zeta\vert^2\pm\sqrt{r} e^{i(\vartheta+\pi)/2}}\,,\qquad \mathrm{Re}\, M_\pm>0\,,
$$
the solution to \eqref{r} is
\begin{equation*}
\begin{split}
v(t)=&\left(-\frac{M_-+M_+}{2 M_-}k_1-\frac{M_--M_+}{2 M_-}k_2\right)e^{-M_-t}+ k_1e^{-M_+t}\\
&+\left(-\frac{M_--M_+}{2 M_-}k_1-\frac{M_-+M_+}{2 M_-}k_2\right)e^{M_-t} +k_2e^{M_+t}
\end{split}
\end{equation*}
for $t\ge 0$ with $k_j\in \R$. Since $v$ must be bounded, $k_1=k_2=0$ and thus $v\equiv 0$. Consequently,
assumptions (m), (n), and (o) from \cite[Theorem~2.18]{Gu_MN} are satisfied and it follows that the operator $-A$
generates an analytic semigroup $\{e^{-tA}\,;\, t\ge 0\}$ on $B_{1,1}^{\alpha}(D)$ (recall that $\alpha\in (4\xi-3,0)\subset (-2,1)$). Similarly, \cite[Remarks~4.2~(b)]{Amann_Teubner} ensures that $-A$ restricted to $W_{2,\mathcal{B}}^{4}(D)$
generates an analytic semigroup $\{e^{-tA}\,;\, t\ge 0\}$ on $L_2(D)$. Notice then that \cite[Proposition~4.13]{Gu_MZ} implies that
$$
\big( B_{1,1}^{\alpha}(D),B_{1,1,\mathcal{B}}^{4+\alpha}(D)\big)_{\theta,1}\doteq B_{1,1,\mathcal{B}}^{4\theta+\alpha}(D)\,,\quad 4\theta\in (0,4)\setminus\{1-\alpha,2-\alpha\}\,,
$$
with $( \cdot,\cdot)_{\theta,1}$ denoting the real interpolation functor. Thus, standard regularizing effects of analytic semigroups \cite[II.Lemma~5.1.3]{LQPP} imply \eqref{6}.
\end{proof}

\begin{proof}[Proof of Theorem~\ref{B}]
To finish off the proof of Theorem~\ref{B} we first recall the continuity of the following embeddings
\bqn\label{5}
B_{1,1,\mathcal{B}}^{4+\alpha}(D)\hookrightarrow B_{1,1,\mathcal{B}}^{s}(D)\hookrightarrow B_{1,1,\mathcal{B}}^{0}(D)\hookrightarrow L_1(D)\hookrightarrow B_{1,1}^{\alpha}(D)\,,\quad s\in (0,4+\alpha)\,,
\eqn
bearing in mind that $\alpha<0$. Now, introducing
$$
h(t):=-\lambda g (u(t))+ a \|\nabla u(t)\|_{L_2(D)}^2\, \Delta u(t)\,,\quad t\in [0,T_m)\,,
$$
we deduce from \eqref{5}, Corollary~\ref{C2.3}, and Corollary~\ref{C2.4} that
\begin{equation}\label{7}
\|h(t)\|_{B_{1,1}^{\alpha}(D)}\le c \,\|h(t)\|_{L_1(D)} \le c\,,\quad t\in [0,T_m)\cap [0,T_0]\,.
\end{equation}
Since $\alpha\in (4\xi-3,0)$ we can fix $\theta\in (0,1)$ and $4\xi_1\in (4\xi,4)\setminus\{3\}$ such that 
\begin{equation*}
4\theta+\alpha>4\xi_1+1>4\xi+1
\end{equation*} 
and, consequently, (see \cite[Section~5]{Amann_Teubner} for instance),
\bqn\label{8}
B_{1,1,\mathcal{B}}^{4\theta+\alpha}(D)\hookrightarrow B_{2,2,\mathcal{B}}^{4\xi_1}(D) \doteq W_{2,\mathcal{B}}^{4\xi_1}(D) \hookrightarrow W_{2,\mathcal{B}}^{4\xi}(D)\,.
\eqn
 Therefore, from \eqref{6}, \eqref{7}, \eqref{8}, and Duhamel's formula
$$
u(t)=e^{-tA}u^0+\int_0^t e^{-(t-s)A} h(s)\,\rd s\,, \quad t\in [0,T_m)\,,
$$
it follows that
\begin{align*}
\|u(t)\|_{W_{2,\mathcal{B}}^{4\xi}(D)}&\le \|e^{-tA}\|_{\mathcal{L}(W_{2,\mathcal{B}}^{4\xi}(D))}\|u^0\|_{W_{2,\mathcal{B}}^{4\xi}(D)}\\
&\qquad + c(\theta) \int_0^t \|e^{-(t-s)A} h(s)\|_{B_{1,1,\mathcal{B}}^{4\theta+\alpha}(D)} \\
& \le c + c(\theta) \int_0^t \|e^{-(t-s)A}\|_{\mathcal{L}(B_{1,1}^{\alpha}(D),B_{1,1,\mathcal{B}}^{4\theta+\alpha}(D))}\|  h(s)\|_{B_{1,1}^{\alpha}(D)}\,\rd s\\
&\le c(\theta)
\end{align*}
for $t\in [0,T_m)\cap [0,T_0]$. We have thus shown \eqref{1a} and the proof of Theorem~\ref{B} is complete according to Theorem~\ref{A}.
\end{proof}

\bibliographystyle{amsplain} 
\bibliography{3dTouchdown}

\end{document}